\documentclass[11pt]{article}
\usepackage{amssymb}
\usepackage{color}
\usepackage{relsize}
\usepackage{amsthm}
\usepackage{amsmath}
\usepackage{amsfonts}

\usepackage{graphicx}

 \newtheorem{thm}{Theorem}[section]
 
 \newtheorem{lem}[thm]{Lemma}
 
 \theoremstyle{definition}
 
 \newtheorem{rem}[thm]{Remark}
 \numberwithin{equation}{section}

\newtheorem{theorem}{Theorem}[section]
\newtheorem{lemma}[theorem]{Lemma}

\newtheorem{proposition}[theorem]{Proposition}

\newcommand{\norm}[1]{\left\lVert#1\right\rVert} 

\theoremstyle{definition}

\theoremstyle{remark}

\begin{document}
\title{On the Constantin-Lax-Majda Model with Convection}
\author{Zhen Lei \footnotemark[1]\ \footnotemark[2]
\and Jie Liu    \footnotemark[1]\ \footnotemark[3]
\and Xiao Ren    \footnotemark[1]\ \footnotemark[4]
}
\renewcommand{\thefootnote}{\fnsymbol{footnote}}
\footnotetext[1]{School of Mathematical Sciences; LMNS and Shanghai
Key Laboratory for Contemporary Applied Mathematics, Fudan University, Shanghai 200433, P. R.China.} \footnotetext[2]{Email: zlei@fudan.edu.cn}
\footnotetext[3]{Email: j\_liu18@fudan.edu.cn}
\footnotetext[4]{Email: xiaoren18@fudan.edu.cn}
\date{\today}
\maketitle

\begin{abstract}
The well-known Constantin-Lax-Majda (CLM) equation, an important toy model of the 3D Euler equations without convection,  can develop finite time singularities \cite{CLM}.  De Gregorio modified the CLM model by adding a convective term \cite{DG}, which is known important for fluid dynamics \cite{HL, Okamoto}. Presented are two results on the De Gregorio model. The first one is the global well-posedness of such a model for general initial data with non-negative (or non-positive) vorticity which is based on a newly discovered conserved quantity. This verifies the numerical observations for such class of initial data. The second one is an exponential stability result of ground states, which is similar to the recent significant work of Jia, Steward and Sverak \cite{JiaSS}, with the zero mean constraint on the initial data being removable. The novelty of the method is the introduction of the new solution space $\mathcal{H}_{DW}$ together with a new basis and an effective inner product of $\mathcal{H}_{DW}$.
\end{abstract}

\maketitle





\section{Introduction}

The classical Constantin-Lax-Majda (CLM) model is
$$\partial_t\omega = \omega H\omega,$$
where $\omega: \mathbb{R}_+ \times \Omega \to \mathbb{R}$ with $\Omega$ being the whole real line $\mathbb{R}$ or the circle $\mathcal{S}^1$,  and $H$ is the Hilbert transform. It is one of the  famous models which are proposed to analyze the potential singularities of 3D Euler equations \cite{CLM}, mimicking the essence of the 3D mechanism and, at the
same time, being feasible for mathematical analysis. On one hand, the blowup mechanism for the CLM model has been well understood by experts (see, for instance, \cite{CLM}). On the other hand, as being pointed out by Okamoto \cite{Okamoto}, the CLM model ignores the role of convection, which we now know is important, see \cite{HL, Okamoto}.

De Gregorio \cite{DG} suggested to include a convective term to the CLM model. The resulting modified equation reads
\begin{equation}\label{DG}
\partial_t\omega + u\partial_\theta \omega = \omega \partial_\theta u,\quad \partial_\theta u = H\omega.
\end{equation}
Here $H$ denotes the Hilbert transform defined as
\begin{equation}\nonumber
H\omega(\theta) = \frac{1}{2\pi}{\rm P.V.}\int_{-\pi}^\pi\cot\frac{\theta - \phi}{2}\omega(\phi)d\phi
\end{equation}
if $\omega$ is defined on the circle $\mathcal{S}^1$ and
\begin{equation}\nonumber
H\omega(\theta) =  \frac{1}{\pi}{\rm P.V.}\int_{-\infty}^\infty\frac{\omega(\phi)}{\theta - \phi}d\phi
\end{equation}
if $\omega$ is defined on the whole line $\mathbb{R}$.
Note that \eqref{DG} is still incomplete, since $u$ is determined by $\omega$ only up to a constant. One needs to impose certain gauge conditions on $u$ such as $\int_{\mathcal{S}^1} u d\theta = 0$ or $u(t,0) \equiv 0$. Solutions to \eqref{DG} under different gauges are equivalent up to translations, see \cite{JiaSS}.

Numerical simulations of De Gregorio model \cite{DG} seem to suggest that there is no blow-up from smooth initial data on $\mathcal{S}^1$. Moreover, as mentioned by authors in \cite{JiaSS}, their numerical experiments seem to suggest that smooth solutions even converge to equilibria of \eqref{DG}(and usually to the ground states $A\sin(\theta+\theta_0)$). On the other hand, finite time singularity of De Gregorio model on the whole line $\mathbb{R}$ has recently been proved for some smooth odd initial data, see \cite{CHH}. We also mention that one can have finite time blow-up for initial data with lower regularity in which the equations are nevertheless still locally well-posed, see \cite{EJ}.

In this paper, we will establish global well-posedness for general initial data with non-negative vorticity $\omega(0, \cdot) \geq 0$ (or non-positive) both on $S^1$ and on $\mathbb{R}$, but without any smallness assumptions. In view of the work in \cite{CHH} where finite time singularities of solutions are proven for a class of sign-changing initial data, the sign condition on the initial data in the following result can not be removed in the whole line case.
\begin{thm}\label{GW}
Let $\Omega$ be the whole line $\mathbb{R}$ or the circle $\mathcal{S}^1$. Let $k\ge 1$ be an integer, and  the initial data have a compact support and satisfy $\omega_{\scriptscriptstyle in}\geq 0$,  $ \sqrt{\omega_{\scriptscriptstyle in}} \in H^k(\Omega)$. Suppose that $u$ satisfies the gauge $\int_{\mathcal{S}^1}u(t, \theta)d\theta \equiv 0$ if $\Omega = \mathcal{S}^1$ or $u(t, 0) \equiv 0$ if $\Omega = \mathbb{R}$.Then the De Gregorio modification of the CLM model \eqref{DG} is globally well-posed in $C([0,T];H^k(\Omega))$ and $$\left\|\sqrt{\omega(t, \cdot)}\right\|_{H^1(\Omega)} = \|\sqrt{\omega_{\scriptscriptstyle in}}\|_{H^1(\Omega)}$$ for all $t \geq 0$.
\end{thm}

\begin{rem} \label{rem-gauge}
Our existence result holds true for different choices of gauges for the velocity $u$. For $\Omega=\mathcal{S}^1$, one could replace the condition $\int_{\mathbb{S}^1}u(t, \theta)d\theta \equiv 0$ with $u(t,0)\equiv 0$. The latter gauge will be convenient for our stability result in Section 3.
\end{rem}
\begin{rem}\label{rem1}
It is clear that $\partial_\theta(\sqrt{\omega_{\scriptscriptstyle in}}) \in L^2$ is a reasonable assumption on the initial data $\omega_{\scriptscriptstyle in}$ if  $\omega_{\scriptscriptstyle in}$ is strictly positive or degenerates at its zeros at an order $\gamma > 1$.  Our key observation here is the a priori conservation of the quantity $\|\partial_\theta(\sqrt{\omega(t, \cdot)})\|_{L^2}$ in time (see Section 2), which seems totally new in the literatures (see \cite{JiaSS} for other known conserved quantities). Note that
$$\partial_t\partial_\theta \omega + u\partial^2_{\theta}\omega = \omega \partial^2_{\theta} u,$$
which has been observed by Jia, Steward and Sverak \cite{JiaSS}. Thus zeros of $\omega(t, \cdot)$ and the values of  $\partial_\theta\omega(t, \cdot)$ at these zeros are transported by $u$, which makes sense of $\|\partial_\theta(\sqrt{\omega})\|_{L^2}$.
\end{rem}
\begin{rem}
Theorem \ref{GW} also holds true for non-positive initial data $\omega_{\scriptscriptstyle in}$, in which case we require that $\omega_{\scriptscriptstyle in}\in L^1(\Omega)$ and $\partial_\theta(\sqrt{-\omega_{\scriptscriptstyle in}})\in L^2(\Omega)$. In fact, we only need to consider $\bar{\omega}=-\omega$, which satisfies:
\begin{equation}\label{DG-reverse}
-\partial_t\bar{\omega} + \bar{u}\partial_\theta\bar{\omega} = \bar{\omega} \partial_\theta\bar{u},\quad \partial_\theta\bar{u} = H\bar{\omega}.
\end{equation}
Note that equation \eqref{DG-reverse} is simply a time reversed version of \eqref{DG} and the proof of Theorem \ref{GW} in Section 2 still works.
\end{rem}

Our second result concerns the recent interesting work of Jia, Steward and Sverak \cite{JiaSS} in which the authors proved, under the mean zero constraint on the initial data, the exponential stability of ground states of the De Gregorio modification of the CLM model \eqref{DG} on the circle for initial data $\theta^{- \gamma}\eta_{\scriptscriptstyle in} \in L^2$, $\frac{3}{2} < \gamma < 2$, where $\eta_{\scriptscriptstyle in} = \omega_{\scriptscriptstyle in} + \sin \theta$. Their proof  involves some deep spectral theories and complex variable methods, together with many novel observations on the structure of the De Gregorio model formulated as a dynamical system. Here we prove a similar exponential stability result to theirs which implies Theorem 1.1 in \cite{JiaSS}, using a direct energy method (the definition of the space $\mathcal{H}_{DW}$ will be introduced after stating the theorem).

\begin{thm}\label{grandstates}
Let $0<\beta<\frac38$ be a given constant and $\omega_{\scriptscriptstyle in}(\theta) = -\sin\theta + \eta_{\scriptscriptstyle in}(\theta)$ with $\theta \in \mathcal{S}^1$. Suppose that $\eta_{\scriptscriptstyle in} \in \mathcal{H}_{DW}$ and $\int_{\mathcal{S}^1} \eta_{\scriptscriptstyle in} d\theta =0$. There exists $\delta_0 > 0$ such that if $\|\eta_{\scriptscriptstyle in}\|_{\mathcal{H}_{DW}} < \delta_0$, then the De Gregorio modification of the CLM model \eqref{DG} under the gauge $u(t,0)\equiv 0$ with initial data $\omega(0, \theta) = \omega_{\scriptscriptstyle in}(\theta)$ is globally well-posed and $\|\omega(t, \cdot) + \sin\theta\|_{\mathcal{H}_{DW}}  \lesssim e^{- \beta t} \|\eta_{\scriptscriptstyle in}\|_{\mathcal{H}_{DW}}$ for all $t \geq 0$.
\end{thm}

There are two main ingredients in our proof. For simplicity, let us take the odd perturbations as an example (the generic perturbations are a little bit more complicated and are treated in Section 3). Firstly, we define a new effective Hilbert space by
\begin{equation}\label{space}
\mathcal{H}_{DW} = \big\{\eta \in H^1(\mathcal{S}^1)\big| \, \eta(0) =  0,\,\, \int_{- \pi}^\pi\frac{|\partial_\theta\eta|^2}{\sin^2\frac{\theta}{2}}d\theta < \infty\big\}.
\end{equation}
The inner product of $(\mathcal{H}_{DW}, g)$ is defined to be
\begin{equation}\label{innerproduct}
\langle\xi, \eta\rangle_g = \frac{1}{4\pi}\int_{- \pi}^\pi\frac{\partial_\theta\xi \partial_\theta\eta}{\sin^2\frac{\theta}{2}}d\theta.
\end{equation}
Secondly, by introducing the following new vectors
\begin{equation}\label{basis}
\widetilde{e}^{(o)}_{k} = \frac{\sin[(k + 1)\theta]}{k + 1} - \frac{\sin(k\theta)}{k},\quad k \geq 1,
\end{equation}
we find that
\begin{equation}\label{orthogonal}
\langle\widetilde{e}^{(o)}_{k}, \widetilde{e}^{(o)}_{l}\rangle_g = \delta_{kl},\quad k, l \geq 1.
\end{equation}
We remark that the inner product in \eqref{innerproduct} and the exact form of the basis vectors in \eqref{basis} are still mysterious to us even though we are fortunate enough to find them accidentally. With the above \textit{accidental} discoveries, we are able to obtain the exponential decay for the linearized equations by using a direct energy estimate, providing a different angle for the highly non-trivial and interesting method of Jia, Steward and Sverak \cite{JiaSS}.

By removing the constraint $\int_{\mathcal{S}^1} \omega_{\scriptscriptstyle in} d\theta =0$, Theorem \ref{grandstates} can be generalized as follows, which will be proved based on the same ideas.
\begin{thm}\label{grandstates'}
Let $0<\beta<\frac38$ be a given constant and $\omega_{\scriptscriptstyle in}(\theta) = -\sin\theta + \zeta_{\scriptscriptstyle in}(\theta)$ with $\theta \in \mathcal{S}^1$. Suppose that $\zeta_{\scriptscriptstyle in} \in \mathcal{H}_{DW}$. There exists $\delta_1 > 0$ such that if $\|\zeta_{\scriptscriptstyle in}\|_{\mathcal{H}_{DW}} < \delta_1$, then the De Gregorio modification of the CLM model \eqref{DG} under the gauge $u(t,0)\equiv 0$ with initial data $\omega(0, \theta) = \omega_{\scriptscriptstyle in}(\theta)$ is globally well-posed and $\|\omega(t, \cdot) + \sin\theta + \alpha (\cos \theta -1)\|_{\mathcal{H}_{DW}}  \lesssim e^{- \beta t} \|\zeta_{\scriptscriptstyle in}\|_{\mathcal{H}_{DW}}$ for all $t \geq 0$, where $\alpha=\int_{\mathcal{S}^1} \zeta_{\scriptscriptstyle in} d\theta$.
\end{thm}

For the first excited state $\sin(2\theta)$, we have the following result on the linearized level, for even initial data.
\begin{thm}\label{excitedSL}
The linearized equation of \eqref{DG} at $-\sin 2\theta$ reads
\begin{equation} \label{L2}
\partial_t\eta = L_{2}\eta,
\end{equation}
with
\begin{equation}\label{LOES}
L_{2}\eta = - \frac{1}{2}\sin(2\theta)\partial_\theta\eta + \cos(2\theta)\eta - \sin(2\theta)\partial_\theta v + 2\cos(2\theta)v,
\end{equation}
where $v$ satisfies $\partial_\theta v=H\eta$ and the gauge $\int_{\mathcal{S}^1} v\, d\theta \equiv0$. For even initial data $\eta(0, \theta) = \sum_{k \geq 1}\eta_{k}\cos(k\theta)$, \eqref{L2} is well-posed in $C([0,T];H^\frac32)$ for any $0<T<\infty$ and satisfies
\begin{equation} \label{L2X}
\frac{d}{dt} \norm{\eta}_X^2 +\frac32 |\eta_1^{(e)}|^2 = 0,\quad \text{for} \quad \eta = \sum_{k \geq 1}\eta_k^{(e)}\cos(k\theta).
\end{equation}
Here $\norm{\eta}_X^2 \overset{\Delta}{=} \sum_{k\neq 2} {g_k^{(e)}} |\eta_k^{(e)}|^2$ with $g_1^{(e)}=1$ and $g_k^{(e)}\sim k^3$. $\eta_2^{(e)}$ grows at most linearly.
\end{thm}

\begin{rem}
This result is analogous to an observation in \cite{JiaSS} that the linearized equation of \eqref{DG} at $-\sin \theta$ has a conserved (semi)norm. The exact forms of $g_k^{(e)}$ will be given in Section 4. We did not find identities like \eqref{L2X} for linearization of \eqref{DG} at higher excited states, i.e. $\sin k\theta$ with $k\ge 3$. Besides, theorem \ref{excitedSL} doesn't hold for general odd initial data either.
\end{rem}

There are some other aspects on the studies of the De Gregorio modification of the CLM model, see for instance, \cite{BKP, CC, CHKLSY, EK, EKW, OSW, W}. The remaining part of this paper is organized as follows: In Section 2, we derive an identity for the new conserved quantity $\big\|\partial_\theta\big(\sqrt{\omega(t, \cdot)}\,\big)\big\|_{L^2}$ and prove Theorem \ref{GW}. We introduce a new basis of functions in Section 3, which leads to the linear stability of the ground state $-\sin \theta$. Then we prove nonlinear stability as stated in Theorem \ref{grandstates} and Theorem \ref{grandstates'}. The last section is devoted to a careful analysis  of the linearized equation at the excited state $-\sin 2 \theta$ for both odd and even data and proving Theorem \ref{excitedSL}.

\section{Global Wellposedness with Non-negative Initial Vorticity}\label{Sec-setup}

\begin{proof}[Proof of Theorem \ref{GW}]
First of all, let us assume that $\sqrt{\omega}$ is smooth enough, for example $\sqrt{\omega} \in C([0, T]; H^2(\Omega))$, where $\omega$ is a solution to \eqref{DG}, with $\int_{\mathcal{S}^1}u(t, \theta)d\theta \equiv 0$ if $\Omega = \mathcal{S}^1$ or $u(t, 0) \equiv 0$ if $\Omega = \mathbb{R}$. We are going to derive some a priori estimates for $\omega$.

By Sobolev imbedding, one has $\sqrt{\omega} \in L^\infty([0, T] \times \Omega)$, which gives that $\omega \in L^\infty([0, T] \times \Omega)$. As a consequence, one has $\omega \in L^\infty([0, T]; H^2(\Omega))$. By the anti-symmetry property of the Hilbert transform, it is clear that
$$\partial_t \int_{\Omega} \omega dx = \int_{\Omega} (-u \partial_\theta\omega + H\omega \omega ) dx = 2\int_{\Omega} H\omega \omega dx = 0.$$
Hence,
\begin{equation} \label{ap1}
\|\omega(t, \cdot)\|_{L^1(\Omega)} = \|\omega_{\scriptscriptstyle in}\|_{L^1}.
\end{equation}

Next, using
\begin{equation} \nonumber
\partial_t \sqrt{\omega} = -u \partial_\theta\sqrt{\omega} + \frac12 \sqrt{w} H\omega\quad {\rm on}\ \{\theta: \omega > 0\},
\end{equation}
we can take derivative on both sides of the above equation to derive that
$$\partial_t \partial_\theta \sqrt{\omega} = -u \partial_\theta^2 \sqrt{\omega} - \partial_\theta u \partial_\theta \sqrt{\omega} + \frac12 \partial_\theta \sqrt{w} H\omega + \frac12 \sqrt{w} \partial_\theta H\omega \quad {\rm on}\ \{\theta: \omega > 0\}.$$
Further calculations give that
\begin{align*}
\frac12 \partial_t (\partial_\theta\sqrt{\omega})^2 &= - \frac12 u \partial_\theta((\partial_\theta \sqrt{\omega})^2) - \partial_\theta u (\partial_\theta \sqrt{\omega})^2 + \frac12 (\partial_\theta \sqrt{\omega})^2 H\omega   + \frac14 \partial_\theta \omega \partial_\theta H\omega \\
&=- \frac12 u \partial_\theta((\partial_\theta \sqrt{\omega})^2) - \frac12 (\partial_\theta \sqrt{\omega})^2 H\omega + \frac14  \partial_\theta \omega  H \partial_\theta\omega.
\end{align*}
Note that the above equation is also true on $\{\theta: \omega = 0\}$, since at such points $\omega$ reaches its minimum. Integrating over $\Omega$ and using the fact $\int (H \partial_\theta \omega)  \partial_\theta \omega d\theta = 0$, we finally arrive at
$$\partial_t \int_{\Omega} \big(\partial_\theta \sqrt{w}\big)^2 dx = 0.$$
Hence, we also have
\begin{equation}\label{ap2}
\|\partial_\theta\sqrt{\omega(t, \cdot)}\|_{L^2(\Omega)} = \|\partial_\theta \sqrt{\omega_{\scriptscriptstyle in}}\|_{L^2(\Omega)}.
\end{equation}
The above argument certainly implies that $\sup_{0\le t \le T}\|\omega\|_{H^1(\Omega)} \leq C_0$ for some constant $C_0$  depending only on $\|\omega_{\scriptscriptstyle in}\|_{L^1}$ and $\|\partial_\theta\sqrt{\omega_{\scriptscriptstyle in}}\|_{L^2}$.

Note that uniqueness of $C([0,T];H^1)$ solution  for \eqref{DG} can be easily obtained by performing an $L^2$ energy estimate (which is similar to the uniqueness of solutions to \eqref{eq-f}. See the Appendix for details). To finish the proof of the theorem, it remains to establish a local existence and uniqueness theory of \eqref{DG} for initial data satisfying the constraints stated in the theorem.  Consider the evolution equation for $f=\sqrt{\omega}$:
\begin{equation} \label{eq-f}
\begin{cases}
\partial_t f = -u \partial_\theta f + \frac{1}{2} f H (f^2),\\
f(t=0,\cdot)=f_{\scriptscriptstyle in} =\sqrt{\omega_{\scriptscriptstyle in}} \ge 0,
\end{cases}
\end{equation}
where $u$ is determined by $\partial_\theta u=f^2$ and the chosen gauge. The construction of local (non-negative) strong solutions to \eqref{eq-f} is similar to that for the vorticity formulation of 3D Euler equations(See \cite{MB} for the standard vanishing viscosity method or particle trajectory method). For completeness, we present a proof in the Appendix.

\end{proof}

\section{Stability of the Ground State}

As has been observed in \cite{JiaSS,OSW}, \eqref{DG} has an infinite number of stationary solutions, of the form $\sin k\theta$, $\forall k\ge 1$ (up to trivial translations and multiplication by constants). We call $\sin\theta$ the ground state and $\sin k\theta$ ($k \geq 2$) the excited states (by translation $\cos k\theta$ are ground state for $k=1$ and excited states for $k\ge 2$).

\subsection{Linearized equation at $\omega = -\sin \theta$}
Consider solutions to \eqref{DG} of the form
$$\omega = -\sin \theta + \eta,\quad u = \sin\theta + v.$$
Clearly, one has
\begin{equation}\label{DGCLM-p}
\begin{cases}
\partial_t\eta + \sin\theta\partial_\theta(\eta + v) - (\eta + v)\cos\theta = \eta \partial_\theta v -  v\partial_\theta\eta \triangleq - [v, \eta],\\[-4mm]\\
H\eta = \partial_\theta v.
\end{cases}
\end{equation}
In \cite{JiaSS}, J. Hao et al. carefully studied the linearized equation for $\eta$, which reads
\begin{equation}\label{DGCLM-l}
\partial_t\eta = L\eta,
\end{equation}
where $L$ is the linear operator defined by
\begin{equation} \label{Lgroundstate}
L\eta = - [\sin\theta, \eta + v]=-\sin \theta \partial_\theta\eta + \cos \theta \eta -\sin\theta H\eta + \cos\theta v.
\end{equation}
They worked with the gauge $v(t,0)=0$ and under the assumption
\begin{equation} \label{0mode}
\int_{\mathcal{S}^1} \eta d\theta=0.
\end{equation}
This assumption is reasonable since $\int_{\mathcal{S}^1} \eta d\theta$ is an invariant both for the linear problem \eqref{DGCLM-l} and for the nonlinear problem \eqref{DG}. We do not need \eqref{0mode} for now but it will be important in Section 3.3. The representation of $L$ on the Fourier side has been computed in \cite{JiaSS}. For odd data one has
\begin{equation} \label{Lonbasis}
Le^{(o)}_{k} = A_ke^{(o)}_{k + 1} + B_ke^{(o)}_{k - 1},\quad k \geq 2,
\end{equation}
where
$$e^{(o)}_k \overset{\Delta}{=} \sin(k\theta),$$
and
\begin{equation} \label{ABformula}
A_k = - \frac{1}{2}(k - 1)(1 - \frac{1}{k}),\quad B_k = \frac{1}{2}(k + 1)(1 - \frac{1}{k}),\quad k \geq 2.
\end{equation}
For $k=1$ one has
$$Le^{(o)}_1 = 0.$$
H. Jia et al. proved exponential decay of $e^{tL}$ in a weighted $L^2$ space based on a study of spectral properties of $L$, see \cite{JiaSS} for more details.
We take a different approach from \cite{JiaSS}, by introducing a sequence of new basis functions. Denote
$$\tilde{e}^{(o)}_k=\frac{e^{(o)}_{k+1}}{k+1}-\frac{e^{(o)}_k}{k}, \quad k\ge 1.$$
Then
\begin{eqnarray*}
L\tilde{e}^{(o)}_k&=&\frac{A_{k+1}}{k+1}e^{(o)}_{k+2}+\frac{B_{k+1}}{k+1}e^{(o)}_k-\frac{A_k}{k}e^{(o)}_{k+1}-\frac{B_k}{k}e^{(o)}_{k-1}\\
&=&-\frac{k^2}{2(k+1)^2}e^{(o)}_{k+2}+\frac{k(k+2)}{2(k+1)^2}e^{(o)}_k+\frac{(k-1)^2}{2k^2}e^{(o)}_{k+1}-\frac{(k+1)(k-1)}{2k^2}e^{(o)}_{k-1}\\
&=&-\frac{k^2(k+2)}{2(k+1)^2}\left(\frac{e^{(o)}_{k+2}}{k+2}-\frac{e^{(o)}_{k+1}}{k+1}\right)+\left[\frac{(k-1)^2(k+1)}{2k^2}-\frac{k^2(k+2)}{2(k+1)^2}\right]\left(\frac{e^{(o)}_{k+1}}{k+1}-\frac{e^{(o)}_{k}}{k}\right)\\
&& + \frac{(k-1)^2(k+1)}{2k^2}\left(\frac{e^{(o)}_{k}}{k}-\frac{e^{(o)}_{k-1}}{k-1}\right)\\
&=&-\frac{k^2(k+2)}{2(k+1)^2}\tilde{e}^{(o)}_{k+1}+\left[\frac{(k-1)^2(k+1)}{2k^2}-\frac{k^2(k+2)}{2(k+1)^2}\right]\tilde{e}^{(o)}_k+\frac{(k-1)^2(k+1)}{2k^2}\tilde{e}^{(o)}_{k-1}
\end{eqnarray*}
i.e.
$$ L\tilde{e}^{(o)}_k = -d_{k+1}\tilde{e}^{(o)}_{k+1}-(d_{k+1}-d_k)\tilde{e}^{(o)}_k+d_k\tilde{e}^{(o)}_{k-1},\quad d_k=\frac{(k-1)^2(k+1)}{2k^2}.$$
Note that $d_1=0$. The above equality holds true for all $k\ge 1$. We can also include even perturbations by introducing
$$e^{(e)}_k=\cos k\theta-1, \quad k\geq 1,$$
and
$$\tilde{e}^{(e)}_k = \frac{\cos (k+1)\theta-1}{k+1} - \frac{\cos k\theta-1}{k}.$$
The constant $-1$ in the definitions are added to make sure that $e_k^{(e)}(0)=\tilde{e}_k^{(e)}(0)=0$. Set $\tilde{e}^{(e)}_0 = \cos \theta -1=e^{(e)}_1$. Similarly we have
$$Le^{(e)}_k = A_ke^{(e)}_{k + 1} + B_ke^{(e)}_{k - 1} - (1-\frac{1}{k})e^{(e)}_1, \quad k \geq 2$$
and $L\tilde{e}^{(e)}_0 = Le^{(e)}_1=0$. It follows that for $k\geq 2$,
\begin{eqnarray*}
L\tilde{e}^{(e)}_k&=&\frac{A_{k+1}}{k+1}e^{(e)}_{k+2}+\frac{B_{k+1}}{k+1}e^{(e)}_k-\frac{A_k}{k}e^{(e)}_{k+1}-\frac{B_k}{k}e^{(e)}_{k-1} + (-\frac{k}{(k+1)^2}+\frac{k-1}{k^2}) e^{(e)}_1\\
&=&-d_{k+1}\tilde{e}^{(e)}_{k+1}-(d_{k+1}-d_k)\tilde{e}^{(e)}_k+d_k\tilde{e}^{(e)}_{k-1} +\frac{k^2-k-1}{k^2(k+1)^2} \tilde{e}^{(e)}_0.
\end{eqnarray*}
Direct computation shows the above holds for $k=1$ as well:
$$L{\tilde{e}^{(e)}_1}=-\frac38 \tilde{e}^{(e)}_2 - \frac38 \tilde{e}^{(e)}_1 -\frac{1}{4} \tilde{e}^{(e)}_0.$$
Hence if we write $\eta =\sum_{k\geq 1} \tilde{\eta}^{(o)}_k \tilde{e}^{(o)}_k + \sum_{k\geq 0} \tilde{\eta}^{(e)}_k \tilde{e}^{(e)}_k$, then \eqref{DGCLM-l} can be written as the following infinite dimensional ODE system
\begin{equation*}
\begin{cases}
\partial_t \tilde{\eta}^{(o)}_k=-d_{k}\tilde{\eta}^{(o)}_{k-1}-(d_{k+1}-d_k)\tilde{\eta}^{(o)}_k+d_{k+1}\tilde{\eta}^{(o)}_{k+1}, \quad k\ge 1,\\
\partial_t \tilde{\eta}^{(e)}_k=-d_{k}\tilde{\eta}^{(e)}_{k-1}-(d_{k+1}-d_k)\tilde{\eta}^{(e)}_k+d_{k+1}\tilde{\eta}^{(e)}_{k+1}, \quad k\ge 1,
\end{cases}
\end{equation*}
where $d_1 \tilde{\eta}^{(o)}_0$ is understood to be 0. For the ``0th mode", we have
$$\partial_t \tilde{\eta}^{(e)}_0 = \sum_{k\geq 1} \frac{k^2-k-1}{k^2(k+1)^2} \tilde{\eta}^{(e)}_k.$$
Hence formally we deduce
\begin{eqnarray} \label{odddecay}
\frac12 \partial_t \sum_{k\geq 1} (\tilde{\eta}^{(o)}_k ) ^2 &\leq & \sum_{k\geq 1} -d_{k}\tilde{\eta}^{(o)}_{k-1} \tilde{\eta}^{(o)}_k-(d_{k+1}-d_k)(\tilde{\eta}^{(o)}_k)^2+d_{k+1} \tilde{\eta}^{(o)}_k \tilde{\eta}^{(o)}_{k+1} \nonumber  \\
&=& \sum_{k\geq 1} -(d_{k+1}-d_{k})(\tilde{\eta}^{(o)}_k)^2 \nonumber \\
&\le &-\frac{3}{8} \sum_{k\geq 1} (\tilde{\eta}^{(o)}_k ) ^2.
\end{eqnarray}
For the even part, we also have
\begin{eqnarray} \label{evendecay}
\frac12 \partial_t \sum_{k\geq 1} (\tilde{\eta}^{(e)}_k ) ^2
&=& \sum_{k\geq 1} -(d_{k+1}-d_{k})(\tilde{\eta}^{(e)}_k)^2\nonumber\\
&\le &-\frac{3}{8} \sum_{k\geq 1} (\tilde{\eta}^{(e)}_k ) ^2.
\end{eqnarray}
Here we have used the fact that for all $k \ge 1$,
\begin{eqnarray*}
d_{k+1}-d_k&=&\frac{k^2(k+2)}{2(k+1)^2}-\frac{(k-1)^2(k+1)}{2k^2}\\
&=&\frac{1}{2}+\frac{k^2-k-1}{2k^2(k+1)^2}\\
&\ge& \frac{3}{8}.
\end{eqnarray*}
It is important to note that $\tilde{\eta}^{(e)}_0$ has no influence on the evolution of other modes. There are a number of ways to make the calculations \eqref{odddecay} and \eqref{evendecay} rigorous (the summations involved may not converge). For instance one may use basic linear semigroup theory as follows. Consider the real Hilbert space $Y$ formally spanned by the basis functions $\tilde{e}^{(o)}_k, \,\, k\geq 1$ and $\tilde{e}^{(e)}_k, \,\,\, k\geq 0$ in which this basis is orthonormal, i.e.
$$Y = \{  \eta=\sum_{k\geq 1} \tilde{\eta}^{(o)}_k \tilde{e}^{(o)}_k + \sum_{k\geq 0} \tilde{\eta}^{(e)}_k \tilde{e}^{(e)}_k \, \big| \, \{\tilde{\eta}^{(o)}_k\}_{k\geq 1}, \{\tilde{\eta}^{(e)}_k\}_{k\geq 0} \in l^2  \}.$$
Then $L$ defines an unbounded closed operator on the Hilbert space $\tilde{Y}\overset{\Delta}{=}Y/\mathbf{R}\tilde{e}^{(e)}_0$. \eqref{odddecay} and \eqref{evendecay} implies, via a direct application of Hille-Yosida theorem that, $L$ generates a strongly continuous semigroup with the desired decay estimate
$$\norm{e^{tL}\eta(0)}_{\tilde{Y}} \le e^{-\frac{3}{8}t} \norm{\eta(0)}_{\tilde{Y}}.$$
We now deal with $\tilde{\eta}^{(e)}_0$ separately.
\begin{eqnarray*}
|\partial_t \tilde{\eta}^{(e)}_0| &=& \big|\sum_{k\geq 1} \frac{k^2-k-1}{k^2(k+1)^2} \tilde{\eta}^{(e)}_k \big|\\
&\lesssim & \norm{\eta}_{\tilde{Y}}.
\end{eqnarray*}

Hence $\tilde{\eta}^{(e)}_0$ converges exponentially to some limit, which we denote by $\tilde{\eta}^{(e)}_0(\infty)$. So far we have shown that a solution to our linearized equation converges exponentially to $\tilde{\eta}^{(e)}_0(\infty) (\cos \theta -1)$ in $Y$. We state it as the following proposition.

\begin{proposition}
For any initial data $\eta_{\scriptscriptstyle in} \in Y$, there exists a finite number $\tilde{\eta}^{(e)}_0(\infty)$ such that
$$\norm{e^{tL}\eta_{\scriptscriptstyle in} -\tilde{\eta}^{(e)}_0(\infty) (\cos \theta -1)}_{Y} \lesssim e^{-\frac38 t}\norm{\eta_{\scriptscriptstyle in}}_{\tilde{Y}}.$$
\end{proposition}

\subsection{Equivalence of norms}
In this section, we point out that the $Y$-norm is actually equivalent to a weighted $\dot{H}^1$ norm. This observation is essential for proving nonlinear stability, since our basis functions $\tilde{e}^{(o)}_k$ and $\tilde{e}^{(e)}_k$ are not helpful for estimating the nonlinear terms. We recall the Hilbert space $\mathcal{H}_{DW}$ defined before:
\begin{equation*}
\mathcal{H}_{DW} = \big\{\eta \in H^1(S^1)\big| \eta(0) =  0,\quad \int_{- \pi}^\pi\frac{|\partial_\theta\eta|^2}{\sin^2\frac{\theta}{2}}d\theta < \infty\big\}.
\end{equation*}
And the corresponding inner product of $(\mathcal{H}_{DW}, g)$ is defined to be
\begin{equation*}
\langle\xi, \eta \rangle_g = \frac{1}{4\pi}\int_{- \pi}^\pi\frac{\partial_\theta\xi\partial_\theta\eta}{\sin^2\frac{\theta}{2}}d\theta.
\end{equation*}
We claim that there is an isometry between $Y$ and $\mathcal{H}_{DW}$, given by the following lemma:
\begin{lem}$\{\tilde{e}^{(o)}_k, k\ge 1\} \cup \{\tilde{e}^{(e)}_l, l\ge 0\}$ is a complete orthonormal basis for $\mathcal{H}_{DW}$.
\end{lem}
\begin{proof} First we notice that
$$\frac{\partial_\theta\tilde{e}^{(o)}_k}{\sin\frac{\theta}{2}} = -2\sin (k+\frac12)\theta, \quad \frac{\partial_\theta\tilde{e}^{(e)}_l}{\sin\frac{\theta}{2}} = 2\cos (l+\frac12)\theta, \quad \forall \, k\ge 1, l\ge 0. $$
Hence
$$\langle\tilde{e}^{(o)}_k,\tilde{e}^{(o)}_l \rangle_g=\delta_{kl},\quad k,l\ge 1,$$
$$\langle\tilde{e}^{(o)}_k,\tilde{e}^{(e)}_l \rangle_g=0,\quad k\ge 1,l\ge 0,$$
$$\langle\tilde{e}^{(e)}_k,\tilde{e}^{(e)}_l \rangle_g=\delta_{kl},\quad k,l\ge 0.$$
It remains to show completeness. Assume that $\xi\in \mathcal{H}_{DW}$, satisfying
$$\langle\xi,\tilde{e}^{(o)}_k\rangle_g=0, \,\,\langle\xi,\tilde{e}^{(e)}_l\rangle_g=0, \,\, \forall \,k\ge 1 , l\ge 0,$$
i.e.,
$$\int_{-\pi}^{\pi}\frac{\partial_\theta\xi}{\sin\frac{\theta}{2}}\sin(k+\frac{1}{2})\theta d\theta=0,\quad \forall k\ge 1.$$
and
$$\int_{-\pi}^{\pi}\frac{\partial_\theta\xi}{\sin\frac{\theta}{2}}\cos(l+\frac{1}{2})\theta d\theta=0,\quad \forall l\ge 0.$$
We note that the first equality holds for $k=0$ as well, since
$$\int_{-\pi}^\pi \partial_\theta\xi d\theta=0.$$
Since $\{\sin(k+\frac{1}{2})\theta, k\ge 0\} \cup \{\cos(l+\frac{1}{2})\theta, l\ge 0\}$ forms a complete basis of $L^2(S^1)$, thus we have
$\partial_\theta\xi=0$, which implies $\xi=0$.
\end{proof}

Hence we can now identify $Y$ as $\mathcal{H}_{DW}$. Clearly there is a continuous embedding $Y = \mathcal{H}_{DW} \hookrightarrow H^1$. Under the condition \eqref{0mode}, proposition 3.1 can be improved using this embedding along with the invariance of $\int_{-\pi}^\pi \eta d\theta$.

\begin{proposition}
For any initial data $\eta_{\scriptscriptstyle in} \in \mathcal{H}_{DW}$ satisfying $\int_{-\pi}^\pi \eta_{\scriptscriptstyle in} d\theta=0$, we have
$$\norm{e^{tL}\eta_{\scriptscriptstyle in}}_{\mathcal{H}_{DW}} \lesssim e^{-\frac38 t}\norm{\eta_{\scriptscriptstyle in}}_{\mathcal{H}_{DW}}.$$
\end{proposition}
\begin{proof}
From \eqref{Lgroundstate} it is easy to see that $\int_{-\pi}^\pi e^{tL}\eta_{\scriptscriptstyle in} d\theta \equiv 0$ is conserved. Passing to limit, we obtain
$$\int_{-\pi}^\pi \tilde{\eta}^{(e)}_0(\infty) (\cos \theta -1) d\theta=0.$$
Hence $\tilde{\eta}^{(e)}_0(\infty)=0$.
\end{proof}

The constraint $\int_{S^1} \eta d\theta =0 $ for $\eta \in \mathcal{H}_{DW}$ is equivalent to
\begin{equation} \label{0condition}
\tilde{\eta}_0^{(e)}=\sum_{k\ge 1} \frac{1}{k(k+1)}\tilde{\eta}_k^{(e)}.
\end{equation}
Recall that the space $\tilde{Y}$ is defined as $Y/\mathbf{R}\tilde{e}^{(e)}_0$ in which the norm is given by
$$\norm{\eta}^2_{\tilde{Y}} = \sum_{k\ge 1}{|\tilde{\eta}_k^{(o)}|^2} + \sum_{l\ge 1}{|\tilde{\eta}_l^{(e)}|^2}.$$
Hence \eqref{0condition} implies that
\begin{equation} \label{tildeY-Hdw}
\norm{\eta}_{\mathcal{H}_{DW}} \lesssim \norm{\eta}_{\tilde{Y}} \le \norm{\eta}_{\mathcal{H}_{DW}}.
\end{equation}
This observation will be useful in the next section.
\subsection{Nonlinear stability}
\begin{proof}[Proof of Theorem \ref{grandstates}]
Consider the nonlinear equation for $\eta = \omega+\sin x$,
\begin{equation} \label{nl_groundstate}
\begin{cases}
\eta_t=L\eta+\partial_\theta v\eta-v\partial_\theta\eta,\\
\partial_\theta v=H\eta, \quad v(t,0)=0.
\end{cases}
\end{equation}
To avoid the technical difficulties caused by the evolution of $\tilde{\eta}_0^{(e)}$, we work with the natural inner product $\tilde{g}$ in $\tilde{Y}$. The discussion in Section 3.1 gives
$${\langle L\eta ,\eta \rangle}_{\tilde{g}} \leq -\frac38 \norm{\eta}_{\tilde{Y}}^2.$$
From \eqref{nl_groundstate}, taking $\tilde{g}$-inner product with $\eta$, we get
\begin{eqnarray*}
\frac12\frac{d}{dt}\|\eta\|_{\tilde{Y}}^2 &=& {\langle L\eta ,\eta \rangle}_{\tilde{g}} + {\langle \partial_\theta v\eta - v \partial_\theta\eta,\eta\rangle}_{\tilde{g}} \\
&\le&-\frac38 \norm{\eta}_{\tilde{Y}}^2+{\langle \partial_\theta v\eta-v\partial_\theta\eta,\eta - \tilde{\eta}_0^{(e)} (\cos \theta -1) \rangle}_g.
\end{eqnarray*}
where $g$ is the inner product in $\mathcal{H}_{DW}$. We estimate the second term as
\begin{eqnarray*}
\langle \partial_\theta v\eta-v\partial_\theta\eta,\eta- \tilde{\eta}_0^{(e)} (\cos \theta -1)\rangle_g&=&\frac{1}{4\pi}\int_{- \pi}^\pi\frac{(\partial^2_{\theta} v\eta-v\partial^2_\theta\eta)(\partial_\theta\eta+\tilde{\eta}_0^{(e)}\sin \theta)}{\sin^2\frac{\theta}{2}}d\theta\\
&=&\frac{1}{4\pi}\int_{- \pi}^\pi\frac{\partial_\theta^2 v\eta\partial_\theta\eta}{\sin^2\frac{\theta}{2}}d\theta-\frac{1}{4\pi}\int_{- \pi}^\pi\frac{v\partial_\theta^2\eta\partial_\theta\eta}{\sin^2\frac{\theta}{2}}d\theta \\
&&+\frac{\tilde{\eta}_0^{(e)}}{2\pi} \int_{-\pi}^{\pi} (\partial_\theta^2 v\eta -v\partial_\theta^2\eta) \frac{\cos \frac{\theta}{2}}{\sin \frac{\theta}{2}} d\theta\\
&=:&\text{I}+\text{II}+\text{III}.
\end{eqnarray*}
Estimate of I:
\begin{eqnarray*}
\text{I} &\lesssim & \left\|\partial_\theta^2 v\right\|_{L^2} \left\|\frac{\eta}{\sin\frac{\theta}{2}}\right\|_{L^\infty} \left\|\frac{\partial_\theta\eta}{\sin\frac{\theta}{2}}\right\|_{L^2}\\
& \lesssim & \left\|\partial_\theta\eta\right\|_{L^2}\|\eta\|^2_{\mathcal{H}_{DW}}\\
& \lesssim & \|\eta\|^3_{\mathcal{H}_{DW}}.
\end{eqnarray*}
For the second line we have used the following estimate
\begin{eqnarray} \label{eta/theta}
\left|\frac{\eta}{\sin\frac{\theta}{2}}\right|&\lesssim& \left|\frac{1}{\theta}\int_0^{\theta}\partial_\theta\eta(\tau)d\tau\right|\nonumber\\
&\lesssim &\frac{1}{\theta}\left\|\frac{\partial_\theta\eta}{\sin\frac{\theta}{2}}\right\|_{L^2}\left(\int_0^{\theta}\sin^2\frac{\tau}{2} d\tau\right)^{\frac 12}\nonumber\\
&\lesssim& \|\eta\|_{\mathcal{H}_{DW}}.
\end{eqnarray}
Estimate of II:
\begin{eqnarray*}
\text{II} &=& -\frac{1}{8\pi} \int_{-\pi}^{\pi}\frac{v\partial_\theta(\partial_\theta\eta)^2}{|\sin\frac{\theta}{2}|^2} d\theta\\
&=& \frac{1}{8\pi} \int_{-\pi}^\pi \frac{\partial_\theta v(\partial_\theta\eta)^2}{|\sin\frac{\theta}{2}|^2} d\theta+ \frac{1}{8\pi} \int_{-\pi}^\pi v(\partial_\theta\eta)^2 \partial_\theta\left(\frac{1}{|\sin\frac{\theta}{2}|^2}\right)d\theta\\
&\lesssim& \norm{\partial_\theta v}_{L^\infty}\norm{\frac{\partial_\theta\eta}{\sin\frac{\theta}{2}}}_{L^2}^2+ \int_{-\pi}^\pi \big| v(\partial_\theta\eta)^2\frac{\cos\frac{\theta}{2}}{\sin^3\frac{\theta}{2}} \big| d\theta\\
&\lesssim& \norm{\partial_\theta\eta}_{L^2} \norm{\eta}_{\mathcal{H}_{DW}}^2+\norm{\frac{v}{\sin\frac{\theta}{2}}}_{L^\infty}\norm{\eta}_{\mathcal{H}_{DW}}^2\\
&\lesssim& \norm{\eta}^3_{\mathcal{H}_{DW}}.
\end{eqnarray*}
For the last inequality we have used Sobolev embedding as follows
\begin{equation} \label{v/theta}
\norm{\frac{v}{\sin\frac{\theta}{2}}}_{L^\infty}\lesssim \norm{\partial_\theta v}_{L^\infty}\lesssim \norm{\partial_\theta^2 v}_{L^2}\lesssim \norm{\partial_\theta \eta}_{L^2}.
\end{equation}
Estimate of III:
\begin{eqnarray*}
\text{III} &\lesssim& \norm{\eta}_{\mathcal{H}_{DW}} \int_{\mathcal{S}^1} \left| \frac{(\eta \partial_\theta H\eta +\partial_\theta\eta H\eta ) \cos \frac{\theta}{2}}{\sin \frac{\theta}{2}} \right| + \left| \frac{v\partial_\theta\eta}{\sin^2 \frac{\theta}{2}} \right|d\theta  \\
&\lesssim& \norm{\eta}_{\mathcal{H}_{DW}} \big( \norm{\frac{\eta}{\sin \frac{\theta}{2}}}_{L^\infty} \norm{\partial_\theta H\eta}_{L^2} + \norm{\eta}_{\mathcal{H}_{DW}} \norm{H\eta}_{L^\infty} + \norm{\frac{v}{\sin \frac{\theta}{2}}}_{L^\infty}  \norm{\eta}_{\mathcal{H}_{DW}} \big)\\
&\lesssim& \norm{\eta}^3_{\mathcal{H}_{DW}}.
\end{eqnarray*}

Combining the above estimates we arrive at
\begin{equation}
\frac{d}{dt}\norm{\eta}^2_{\tilde{Y}}\le -\frac{3}{4}\norm{\eta}^2_{\tilde{Y}}+C\norm{\eta}^3_{\mathcal{H}_{DW}},
\end{equation}
for some constant $C>0$. Remembering that we are working under the condition $\int_{\mathcal{S}^1} \eta d\theta=0$, \eqref{tildeY-Hdw} is valid. Theorem \ref{grandstates} now follows easily from the above energy estimate.
\end{proof}

Theorem \ref{grandstates'} can be proved in a similar way,  with a few additional terms to be estimated.

\begin{proof}[Proof of Theorem \ref{grandstates'}]
Let $\eta(t,\theta) = \zeta(t,\theta) + \alpha(\cos \theta -1) = \omega(t,\theta) + \sin \theta + \alpha(\cos \theta -1)$, where $\alpha = \int_{\mathcal{S}^1} \zeta_{\scriptscriptstyle in} d\theta$. Then we have $\int_{\mathcal{S}^1} \eta  d\theta =0$ and $\eta$ satisfies the following evolution equations
\begin{equation*}
\begin{cases}
\eta_t=L\eta +\alpha ((1-\cos \theta) (\partial_\theta\eta+\partial_\theta v) -\sin \theta (\eta+v)) +\partial_\theta v\eta-v\partial_\theta\eta,\\
\partial_\theta v=H\eta, \quad v(t,0)=0.
\end{cases}
\end{equation*}
Taking $\tilde{g}$-inner product with $\eta$ and applying the estimates in the previous proof, we get
\begin{eqnarray*}
\frac12\frac{d}{dt}\|\eta\|_{\tilde{Y}}^2 &=& {\langle L\eta ,\eta \rangle}_{\tilde{g}} + \alpha{\langle (1-\cos \theta) (\partial_\theta\eta+\partial_\theta v) -\sin \theta (\eta+v), \eta \rangle}_{\tilde{g}} + {\langle \partial_\theta v\eta - v \partial_\theta\eta,\eta\rangle}_{\tilde{g}} \\
&\le& -\frac{3}{4}\norm{\eta}^2_{\tilde{Y}}+ \alpha{\langle (1-\cos \theta) (\partial_\theta\eta+\partial_\theta v) -\sin \theta (\eta+v), \eta-\tilde{\eta}_0^{(e)}(\cos\theta -1) \rangle}_{g} \\
&& +C\norm{\eta}^3_{\mathcal{H}_{DW}}
\end{eqnarray*}
Denote the second term on the right hand side by $S$. We handle it using integration by parts, \eqref{eta/theta} and \eqref{v/theta},
\begin{eqnarray*}
S&=&\frac{\alpha}{4\pi}\int_{\mathcal{S}^1} \left[(1-\cos \theta) (\partial_\theta^2\eta+\partial_\theta^2 v) - \cos \theta (\eta+v)\right]\frac{\partial_\theta \eta+\tilde{\eta}_0^{(e)}\sin \theta}{\sin^2 \frac{\theta}{2}} d\theta\\
&=& \frac{\alpha\tilde{\eta}_0^{(e)}}{2\pi} \int_{\mathcal{S}^1} (\partial_\theta^2\eta+\partial_\theta^2 v) \sin\theta d\theta - \frac{\alpha}{4\pi} \int_{\mathcal{S}^1} (\eta+v) \frac{\partial_\theta\eta\cos \theta}{\sin^2 \frac{\theta}{2}} d\theta - \frac{\alpha\tilde{\eta}_0^{(e)}}{2\pi} \int_{\mathcal{S}^1} (\eta+v)\frac{\cos \frac{\theta}{2}}{\sin \frac{\theta}{2}}d\theta\\
&\lesssim& |\alpha|\tilde{\eta}_0^{(e)} (\norm{\partial_\theta\eta}_{L^2} + \norm{H\partial_\theta\eta}_{L^2})+ (|\alpha| \norm{\eta}_{\mathcal{H}_{DW}} + |\alpha|\tilde{\eta}_0^{(e)})\big(\norm{\frac{\eta}{\sin \frac{\theta}{2}}}_{L^\infty}+ \norm{\frac{v}{\sin \frac{\theta}{2}}}_{L^\infty}\big)\\
&\lesssim& |\alpha| \norm{\eta}_{\mathcal{H}_{DW}}^2
\end{eqnarray*}
It remains to use the estimate $|\alpha|+\norm{\eta(0,\cdot)}_{\mathcal{H}_{DW}} \lesssim \norm{\zeta_{\scriptscriptstyle in}}_{\mathcal{H}_{DW}}$ to finish the proof.
\end{proof}

\begin{rem}
In our proof of Theorem \ref{grandstates'}, the smallness of $\alpha$ played an important role. Based on the numerical behaviour of \eqref{DG}, we conjecture that the shifted ground states $-\sin \theta -\alpha(\cos\theta-1)$ are exponentially stable for all $\alpha \in \mathbf{R}$ in some suitable spaces. For now, we only know this is true in $\mathcal{H}_{DW}$ for $|\alpha|$ small.
\end{rem}

\section{Well-posedness of the Linearized Equation at $-\sin 2\theta$}
Consider solutions of the form
$$w=-\sin2\theta+\eta,\quad u=\frac 12\sin2\theta+v.$$
Then the linearized equation for $\eta$ reads
\begin{eqnarray}
\begin{cases}
\eta_t+\frac12\sin2\theta\partial_\theta\eta-\cos2\theta\eta+\sin2\theta H\eta-2\cos2\theta v=0 ,&\\
\partial_\theta v=H\eta .&
\end{cases}
\end{eqnarray}Denote
$$L\eta=-[\sin2\theta, \frac{\eta}{2}+v],$$
then
$$\eta_t=L\eta.$$
As before, we calculate the linearized operator $L$ on the Fourier side. For odd data:
$$Le^{(o)}_k=A_k e^{(o)}_{k+2}+B_k e^{(o)}_{k-2}, \quad e^{(o)}_k=\sin k\theta , \quad k\ge 2, $$and$$Le^{(o)}_1=-\frac{1}{4}e^{(o)}_3+\frac{3}{4}e^{(o)}_1,$$
where the coefficients
$$A_k=-\frac{(k-2)^2}{4k},\quad B_k=\frac{(k+2)(k-2)}{4k}$$
are different from those in Section 3.1. Similarly, for the even data we have
$$Le^{(e)}_k=A_k e^{(e)}_{k+2}+B_k e^{(e)}_{k-2}, \quad e^{(e)}_k=\cos k\theta , \quad k\ge 2,$$
and
$$Le^{(e)}_1=-\frac{1}{4}e^{(e)}_3-\frac{3}{4}e^{(e)}_1.$$
Assume $\eta=\sum_{k\ge 1}\eta^{(o)}_k e^{(o)}_k+\eta^{(e)}_k e^{(e)}_k$ which satisfies $\int_{S^1} \eta d\theta =0$, then we have\begin{eqnarray*}
\partial_t \eta^{(o)}_1(t)&=&\frac34\eta^{(o)}_1+\frac{5}{12}\eta^{(o)}_3, \quad \partial_t \eta^{(o)}_2(t)=\frac34\eta^{(o)}_4,\\
\partial_t \eta^{(o)}_k(t)&=&A_{k-2}\eta^{(o)}_{k-2}+B_{k+2}\eta^{(o)}_{k+2},\quad k\ge 3.
\end{eqnarray*}
and
\begin{eqnarray*}
\partial_t \eta^{(e)}_1(t)&=&-\frac34\eta^{(e)}_1+\frac{5}{12}\eta^{(e)}_3, \quad \partial_t \eta^{(e)}_2(t)=\frac34\eta^{(e)}_4,\\
\partial_t \eta^{(e)}_k(t)&=&A_{k-2}\eta^{(e)}_{k-2}+B_{k+2}\eta^{(e)}_{k+2},\quad k\ge 3.
\end{eqnarray*}
Define $X \subseteqq L^2$ to be the real Hilbert space with the following inner product
$$\langle\eta, \xi\rangle_X = \sum_{k \geq 1}g^{(o)}_k\eta^{(o)}_k\xi^{(o)}_k+g^{(e)}_k\eta^{(e)}_k\xi^{(e)}_k,$$
where $g^{(o)}_k, g^{(e)}_k$ are to be determined later. For odd data, one has
\begin{eqnarray}\nonumber
&&\langle L\eta, \eta\rangle_X \\\nonumber
&=&  g^{(o)}_1\eta^{(o)}_1(\frac34\eta^{(o)}_1+\frac{5}{12}\eta^{(o)}_3)+\frac34 g^{(o)}_2\eta^{(o)}_2\eta^{(o)}_4+\sum_{k \geq 3}g^{(o)}_k\eta^{(o)}_k(A_{k - 2}\eta^{(o)}_{k - 2} + B_{k + 2}\eta^{(o)}_{k + 2}) \\\nonumber
&=& \frac{3}{4}g^{(o)}_1(\eta^{(o)}_1)^2+\sum_{k\ge 1}(g^{(o)}_kB_{k+2}+g^{(o)}_{k+2}A_k)\eta^{(o)}_k\eta^{(o)}_{k+2}.
\end{eqnarray}
Unless we set $g^{(o)}_{2k-1}=0, \,\, \forall k\geq 1$ and
\begin{equation} \label{gorecursive}
g^{(o)}_{2k}B_{2k+2}+g^{(o)}_{2k+2}A_{2k}=0,\quad k \geq 1,
\end{equation}
$\langle L\eta, \eta\rangle_X$ changes sign in general. Hence there seems no natural conserved (or decreasing) norm for odd perturbation at $-\sin 2\theta$. This accounts for the numerically observed instability of $-\sin 2\theta$, as mentioned in \cite{JiaSS}.

We remark that if we only consider initial data of the form $\eta = \sum_{k\geq 1} \eta^{(o)}_{2k}e^{(o)}_{2k}$, a conserved (semi)norm can be found using \eqref{gorecursive}. Furthermore, our methods in Section 3 can be easily adapted to prove exponential decay for such data. More precisely we set
$$\tilde{e}^{(o)}_{2k} = \frac{e^{(o)}_{2k+2}}{2k+2} - \frac{e^{(o)}_{2k}}{2k},\quad k\ge 1,$$
and compute $L$ for such basis vectors. This will lead to exponential decay in the corresponding weighted $\dot{H}^1$ space
$$\mathcal{H}_{DW} = \left\{\eta=\sum_{k\geq 1} \eta^{(o)}_{2k}e^{(o)}_{2k}\Big|\int_{\mathcal{S}^1} \frac{|\partial_\theta\eta|^2}{|\sin \theta|^2} d\theta\right\}.$$
\\

Now we turn to consider even data $\eta = \sum_{k \ge 1}\eta_k^{(e)} e_k^{(e)}$, for which the situation is quite different. We prove Theorem \ref{excitedSL} as follows.
\begin{eqnarray}\label{L2energy1}\nonumber
&&\langle L\eta, \eta\rangle_X \\\nonumber
&=&  -g^{(e)}_1\eta^{(e)}_1(\frac34\eta^{(e)}_1+\frac{5}{12}\eta^{(e)}_3)+\frac34 g^{(e)}_2\eta^{(e)}_2\eta^{(e)}_4+\sum_{k \geq 3}g^{(e)}_k\eta^{(e)}_k(A_{k - 2}\eta^{(e)}_{k - 2} + B_{k + 2}\eta^{(e)}_{k + 2}) \\
&=& -\frac{3}{4}g^{(e)}_1(\eta^{(e)}_1)^2+\sum_{k\ge 1}(g^{(e)}_kB_{k+2}+g^{(e)}_{k+2}A_k)\eta^{(e)}_k\eta^{(e)}_{k+2}.
\end{eqnarray}
Let $g^{(e)}_1 =g^{(e)}_4= 1$, $g^{(e)}_2 = 0$ and set
$$g^{(e)}_kB_{k+2}+g^{(e)}_{k+2}A_k=0,\quad k \geq 1.$$
It is easy to check that
$$
g^{(e)}_k\sim k^3, \quad k\to \infty.\\
$$
The (semi)norm defined by $g^{(e)}_k$ will be decreasing for even data. More precisely from \eqref{L2energy1}  we have
\begin{equation} \label{L2energy}
\langle L\eta, \eta\rangle_X= -\frac{3}{4}(\eta^{(e)}_1)^2.
\end{equation}
\eqref{L2energy} clearly implies Theorem \ref{excitedSL}.

\begin{rem}Here we are not able to prove exponential decay, so nonlinear stability cannot be deduced. However, according to our own numerical experiments, solution to \eqref{DG} with initial data of the form $-\sin 2 \theta + \epsilon \cos\theta$ converges to some multiple of $-\sin 2\theta$ instead of ground states. This interesting phenomenon remains to be investigated in future works.
\end{rem}

\section{Appendix}

Let us present a proof for the local wellposedness theory for \eqref{eq-f}, for completeness. We claim the following:
\begin{lemma}
Let $k\ge 1$ be an integer, and $\Omega = \mathcal{S}^1$ or $\mathbb{R}$. For any initial data $f_{\scriptscriptstyle in} \in H^k(\Omega)$ with a compact support, there exists a time $T=T(k,\|f_{\scriptscriptstyle in}\|_{H^1(\Omega)}) > 0$ and a unique solution $f \in C([0,T];H^k(\Omega))$ to \eqref{eq-f}, with $\partial_t f \in C([0,T];H^{k-1}(\Omega))$. Moreover, $f$ satisfies the following identities for all $0\le t \le T$,
\begin{equation*}
\|f\|_{L^2(\Omega)} = \|f_{\scriptscriptstyle in}\|_{L^2(\Omega)},
\end{equation*}
 and
 \begin{equation*}
\|f\|_{H^1(\Omega)} = \|f_{\scriptscriptstyle in}\|_{H^1(\Omega)}.
\end{equation*}
Besides, if $f_{\scriptscriptstyle in} \ge 0$, then the solution $f \ge 0$ for $0\le t \le T$.
\end{lemma}
\begin{rem}
In fact, one can even obtain similar results for fractional $k$ by using Leibnitz rules for fractional derivatives, say, the Li's law or the classical Kato-Ponce inequalities. See \cite{DL} for a complete presentation.
\end{rem}
\begin{proof}
\begin{itemize}
\item \textbf{Uniqueness.} Assume that $f,g \in C([0,T];H^1)$ solves \eqref{eq-f} with the same initial data $f_{\scriptscriptstyle in}$. By $L^2$ energy estimate, we have
\begin{align*}
\frac{d}{dt} \|f-g\|_{L^2}^2 \lesssim \|f-g\|_{L^2}^2 (\|f\|_{C([0,T];H^1)}^2+\|g\|_{C([0,T];H^1)}^2).
\end{align*}
By Gronwall's inequality, for $0\le t \le T$,
$$\|f-g\|_{L^2} \equiv 0.$$
\item \textbf{Existence.}
\emph{Step 1.} First, we work with smooth initial data $f_{\scriptscriptstyle in} \in H^\infty = \cap_{N\ge 1} H^N$. We use the following iterative scheme to approximate the solution of \eqref{eq-f}:
\begin{equation*}
\begin{cases}
\partial_t f^{(n+1)} = - u^{(n)} \partial_\theta f^{(n+1)} + \frac12 f^{(n+1)} H((f^{(n)})^2), \\
f^{(n+1)}(t=0,\cdot) = f_{\scriptscriptstyle in}.
\end{cases}
\end{equation*}
At each stage $n$, $f^{(n+1)}$ can be solved using the method of characteristics, and is clearly smooth for all times. By energy estimates, we have
\begin{align*}
   \frac{d}{dt} \| f^{(n+1)}\|_{H^1}^2 &\leq  C_\star\| H ((f^{(n)})^2)\|_{H^1} \| f^{(n+1)}\|_{H^1}^2 \\
  &\leq C_\star\|f^{(n)}\|_{H^1}^2 \| f^{(n+1)}\|_{H^1}^2,
\end{align*}
where $C_\star > 1$ is an absolute positive constant whose meaning may change from line to line (in the later $C_\star$ may depend on $k$). Gronwall's inequality gives
\begin{align*}
\| f^{(n+1)}\|_{H^1}^2 \le \| f_{\scriptscriptstyle in}\|_{H^1}^2 e^{C_\star\int_0^t \|f^{(n)}\|_{H^1}^2 ds}.
\end{align*}
 An induction argument on $n$ gives, for $0\le t\le T_1=1/(2eC_\star\|f_{\scriptscriptstyle in}\|_{H^1}^2)$,
\begin{align*}
\|f^{(n+1)}\|_{H^1}^2 \le 2e\|f_{\scriptscriptstyle in}\|_{H^1}^2.
\end{align*}
More generally, using the classical Gagliardo-Nirenberg inequality, we have
\begin{align*}
\frac{d}{dt} \|f^{(n+1)}\|_{H^k}^2 &\leq C_\star\|f_{\scriptscriptstyle in}\|_{H^1}^2 \|f^{(n+1)}\|_{H^k}^2 + C_\star\|f_{\scriptscriptstyle in}\|_{H^1}^2 \|f^{(n)}\|_{H^k} \|f^{(n+1)}\|_{H^k}\\
&\leq C_\star\|f_{\scriptscriptstyle in}\|_{H^1}^2 \|f^{(n+1)}\|_{H^k}^2 + C_\star\|f_{\scriptscriptstyle in}\|_{H^1}^2 \|f^{(n)}\|_{H^k}^2.
\end{align*}
Gronwall's inequality gives
\begin{align*}
\| f^{(n+1)}\|_{H^k}^2 \le (\|f_{\scriptscriptstyle in}\|_{H^k}^2 + C_\star \|f_{\scriptscriptstyle in}\|_{H^1}^2 \int_0^t\|f^{(n)}\|_{H^k}^2ds ) e^{C_\star \int_0^t \|f^{(n)}\|_{H^1}^2 ds}
\end{align*}
Induction on $n$ gives, for $0\le t \le T_k = 1/(2eC_\star \|f_{\scriptscriptstyle in}\|_{H^1}^2)$,
\begin{align*}
\|f^{(n+1)}\|_{H^k}^2 \le 2e \|f_{\scriptscriptstyle in}\|_{H^k}^2.
\end{align*}
Hence $f^{(n)}$ is uniformly bounded in $C([0,T_k];H^k)$ for $n = 0,1,2,\cdots$.
\\
Next we claim that $f^{(n)}$ is a Cauchy sequence in $C([0,T_0],L^2)$ for some $T_0>0$. Note that
\begin{align*}
\begin{cases}
\partial_t (f^{(n+2)} - f^{(n+1)}) = - (u^{(n+1)}- u^{(n)}) \partial_\theta f^{(n+2)} - u^{(n)} \partial_\theta(f^{(n+2)} - f^{(n+1)}) \\
\qquad  \qquad \qquad \qquad \qquad + \frac12 (f^{(n+2)} - f^{(n+1)}) H((f^{(n+1)})^2) \\
\qquad  \qquad \qquad \qquad \qquad + \frac12 (f^{(n+1)}) H((f^{(n+1)})^2-(f^{(n)})^2) \\
(f^{(n+2)} - f^{(n+1)})(t=0,\cdot) = 0
\end{cases}
\end{align*}
Hence $L^2$ estimate gives
\begin{align*}
\frac{d}{dt} \|f^{(n+2)} - f^{(n+1)}\|_{L^2} \le C_\star (\|f^{(n+2)} - f^{(n+1)}\|_{L^2}+ \|f^{(n+1)} - f^{(n)}\|_{L^2}) \|f_{\scriptscriptstyle in}\|_{H^1}^2
\end{align*}
Gronwall's inequality gives
\begin{align*}
\|f^{(n+2)} - f^{(n+1)}\|_{L^2} \le C_\star \|f_{\scriptscriptstyle in}\|_{H^1}^2 \int_0^t  \|f^{(n+1)} - f^{(n)}\|_{L^2} ds \,\,e^{C_\star t \|f_{\scriptscriptstyle in}\|_{H^1}^2}
\end{align*}
Thus, for $0\le t\le T_0=1/(2eC_\star\|f_{\scriptscriptstyle in}\|_{H^1}^2)$, we have
\begin{align*}
\sup_{0\le t\le T_0}\|f^{(n+2)} - f^{(n+1)}\|_{L^2} \le \frac12 \sup_{0\le t\le T_0} \|f^{(n+1)} - f^{(n)}\|_{ L^2}
\end{align*}
This proves the claim. Let us denote the limit function by $f$, then $f\in C([0,\tilde{T}_k];L^2)\cap L^\infty([0,\tilde{T}_k]; H^k)$, with $\tilde{T}_k = \min\{T_0,T_1,T_k\}$. Moreover, we have, for $0\le t \le \tilde{T}_k$,
\begin{align*}
\|f\|_{H^k}^2 \le 2e \|f_{\scriptscriptstyle in}\|_{H^k}^2.
\end{align*}
 It is easy to check that $f$ satisfies the equation \eqref{eq-f} in the sense of distributions. Using the equation \eqref{eq-f}, we have $\partial_t f \in L^\infty([0,\tilde{T}_k];H^{k-1})$ and $f \in C([0,\tilde{T}_k];H^{k-1})$.

According to the proof of \eqref{ap1} and \eqref{ap2}, we have the identities
\begin{equation} \label{id1}
\|f\|_{L^2(\Omega)} = \|f_{\scriptscriptstyle in}\|_{L^2(\Omega)},
\end{equation}
 and
 \begin{equation} \label{id2}
\|f\|_{H^1(\Omega)} = \|f_{\scriptscriptstyle in}\|_{H^1(\Omega)}.
\end{equation}
for $0\le t \le \tilde{T}_k$. We emphasize here that $\tilde{T}_k$ depends only on $k$ and $\|f_{\scriptscriptstyle in}\|_{H^1}$.

\emph{Step 2.} Next, given a general initial data $f_{\scriptscriptstyle in} \in H^k$, we approximate $f_{\scriptscriptstyle in}$ with $P_{\le N} f_{\scriptscriptstyle in} $ and prove a $H^k$ stability result for \eqref{eq-f}. Here $N$ is a dyadic number (i.e., $\log_2 N$ is an integer) and $P_{\le N}$ is the standard Littlewood-Paley projection operator.
 Note that$$\|P_{\le N} f_{\scriptscriptstyle in}\|_{H^k} \le \|f_{\scriptscriptstyle in}\|_{H^k}.$$

 Since $P_{\le N} f_{\scriptscriptstyle in} \in H^\infty$, by \emph{Step 1}, there exists $T=\tilde{T}_{k+3}>0$ depending on $\| f_{\scriptscriptstyle in}\|_{H^1}$ and $k$, and a sequence of solutions $f_N\in C([0,T];H^{k+2})$ with the initial data $P_{\le N} f_{\scriptscriptstyle in}$, which satisfy \eqref{id1} and \eqref{id2} with $f_{\scriptscriptstyle in}$ being replaced by $P_{\leq N}f_{\scriptscriptstyle in}$. We are going to take the limit $N \to \infty$.

 Let $N'>N$. By energy estimate, we have
\begin{align*}
\frac{d}{dt} \|f_N -f_{N'}\|_{L^2}^2 \le C_\star \|f_N -f_{N'}\|_{L^2}^2 \|f_{\scriptscriptstyle in}\|_{H^1}^2.
\end{align*}
This implies that for $0\le t \le T$,
\begin{align*}
\|f_N -f_{N'}\|_{L^2}^2 \le C_\star \|P_{\le N}f_{\scriptscriptstyle in} - P_{\le {N'}} f_{\scriptscriptstyle in}\|_{L^2}^2.
\end{align*}
More generally, we have
\begin{align*}
&\quad\frac{d}{dt} \|f_N-f_{N'}\|_{H^k}^2\\
&\le C_\star \|f_{\scriptscriptstyle in}\|_{H^k}^2 \|f_N-f_{N'}\|_{H^k}^2\\
&\quad + C_\star \|f_{\scriptscriptstyle in}\|_{H^k} \|P_{\le {N}} f_{\scriptscriptstyle in}\|_{H^{k+1}} \|f_N-f_{N'}\|_{L^2} \|f_N-f_{N'}\|_{H^k}\\
&\le C_\star \|f_{\scriptscriptstyle in}\|_{H^k}^2 \|f_N-f_{N'}\|_{H^k}^2\\
&\quad + C_\star \|f_{\scriptscriptstyle in}\|_{H^k} N\|P_{\le N} f_{\scriptscriptstyle in}\|_{H^{k}} \|P_{\le N}f_{\scriptscriptstyle in} - P_{\le {N'}} f_{\scriptscriptstyle in}\|_{L^2} \|f_N-f_{N'}\|_{H^k}
\end{align*}

Note that as $N,N'\to \infty$,
$$N\|P_{\le N}f_{\scriptscriptstyle in} - P_{\le N'} f_{\scriptscriptstyle in}\|_{L^2} \leq \|P_{N \leq \cdot \le N'}f_{\scriptscriptstyle in}\|_{H^1} \to 0.$$
Using Gronwall's inequality we obtain
\begin{align*}
\sup_{0\le t \le T}\|f_N-f_{N'}\|_{H^k}^2\to 0
\end{align*}
Hence $f_N$ converges strongly in $C([0,T];H^k)$ to a limit function $f\in C([0,T];H^k)$ which satisfies the equation \eqref{eq-f} with initial data ${f_{\scriptscriptstyle in}}$. The identities \eqref{id1} and \eqref{id2} for $f_N$ and $P_{\leq N}f_{\scriptscriptstyle in}$ can be passed to the limit.

Finally, we remark that if $f_{\scriptscriptstyle in}$ is non-negative, then using the method of characteristics, $f$ is also non-negative.
\end{itemize}
\end{proof}

\section*{Acknowledgement}
The author was also in part supported by NSFC (grant No. 11725102) and National Support Program for Young Top-Notch Talents.


\frenchspacing
\bibliographystyle{plain}

\begin{thebibliography}{99}

\bibitem{BKP} M. Bauer, B. Kolev and S. Preston, \textit{Geometric investigations of a vorticity model equation}. J. Differ. Eq. 260(1), 478-516 (2016)

\bibitem{CC} A. Castro and D. Cordoba, \textit{Infinite energy solutions of the surface quasi-geostrophic equation}. Adv. Math. 225(4), 1820-1829 (2010)

\bibitem{CHH} J. Chen, T. Y. Hou, D. Huang, \textit{On the Finite Time Blowup of the De Gregorio Model for the 3D Euler Equation}. arXiv:1905.06387(2019).

\bibitem{CHKLSY} K. Choi, T. Y. Hou,  A. Kiselev, G. Luo, V. Sverak, and Y. Yao, \textit{On the finite-time blowup of a one-dimensionalmodel for the three-dimensional axisymmetric Euler equations}. Commun. Pure Appl. Math. (2017) (online), see also arXiv:1407.4776

\bibitem{CLM} P. Constantin, P. D. Lax and A. J. Majda, \textit{A simple one-dimensional model for the three-dimensional vorticity equation}. Commun. Pure Appl. Math. 38, 715-724 (1985)

\bibitem{DG} S. De Gregorio, \textit{A partial differential equation arising in a 1D model for the 3D vorticity equation}. Math. Methods Appl. Sci. 19(15), 1233-1255 (1996)

\bibitem{EJ} T.M. Elgindi and I.-J. Jeong, \textit{On the effects of advection and vortex stretching}. arXiv:1701.04050 (2017)

\bibitem{EK} J. Escher and B. Kolev, \textit{Right-invariant Sobolev metrics of fractional order on the diffeomorphism group of the circle}. J. Geom. Mech. 6(3), 335--372 (2014)

\bibitem{EKW} J. Escher, B. Kolev, M. Wunsch, \textit{The geometry of a vorticity model equation}. Commun. Pure Appl. Anal. 11(4), 1407--1419 (2012)

\bibitem{HL} T. Y. Hou and Z. Lei, \textit{On the stabilizing effect of convection in three-dimensional incompressible flows}. Comm. Pure Appl. Math. 62 (2009), no. 4, 501-564.

\bibitem{JiaSS} H. Jia, S. Steward and V. Sverak,  \textit{On the De Gregorio modification of the
Constantin-Lax-Majda model}. To appear in Archive for Rational Mechanics and Analysis. DOI: 10.1007/s00205-018-1298-1

\bibitem{DL} D. Li, \textit{On Kato-Ponce and fractional Leibniz}. Rev. Mat. Iberoam. 35 , 23-100 (2019).

\bibitem{MB}Andrew J. Majda and Andrea L. Bertozzi, \textit{Vorticity and incompressible flow}. Cambridge University Press (2002).

\bibitem{Okamoto} H. Okamoto, \textit{Blow-up problems in the strained vorticity dynamics and critical exponents}. J. Math. Soc. Japan 65 (2013), no. 4, 1079-1099.

\bibitem{OSW} H. Okamoto, T. Sakajo and M. Wunsch, \textit{On a generalization of the Constantin-Lax-Majda equation}. Nonlinearity 21(10), 2447--2461 (2008)

\bibitem{W} M. Wunsch, \textit{The generalized Constantin-Lax-Majda equation revisited}. Commun. Math. Sci. 9(3), 929--936 (2011)

\end{thebibliography}

\end{document}